\documentclass[11pt]{article}
\usepackage[utf8]{inputenc}
\usepackage{amsmath,amssymb,amsthm}
\usepackage{geometry}
\usepackage{hyperref}
\usepackage{enumitem}
\usepackage{verbatim}
\usepackage{listings}

\title{A Conditional Reduction of the Rational Hodge Conjecture for Threefolds\\
and Deformation-Theoretic Verifications in Several Families}
\author{Karim Mansour\thanks{Email: \texttt{kahmed.math.coder@gmail.com}}}
\date{\today}

\theoremstyle{plain}
\newtheorem{theorem}{Theorem}[section]
\newtheorem{proposition}[theorem]{Proposition}
\newtheorem{lemma}[theorem]{Lemma}

\theoremstyle{definition}
\newtheorem{definition}[theorem]{Definition}
\newtheorem{remark}[theorem]{Remark}

\numberwithin{equation}{section}

\begin{document}

\maketitle

\begin{abstract}
We formulate a concrete geometric approximation hypothesis (Hypothesis BB) asserting that codimension-$2$ Hodge classes on a smooth projective threefold can be realized as specializations of families whose general members are complete-intersection curves. We prove that Hypothesis BB implies the (rational) Hodge conjecture for the threefold. We then give deformation-theoretic sufficient criteria (cohomology-vanishing and surjectivity conditions) which imply Hypothesis BB, and we prove these criteria hold for the class of a line on a \emph{general} quintic threefold containing that line. We further formulate and prove several propositions showing that, under natural Noether–Lefschetz / unobstructedness hypotheses, Hypothesis BB holds \emph{generically} in families of Calabi–Yau and Fano threefolds; these propositions reduce the problem to checkable conditions (normal-bundle cohomology, surjectivity of restriction maps). Finally we include a Macaulay2 appendix with scripts to verify the required cohomology and splitting conditions in explicit examples.
\end{abstract}

\tableofcontents

\section{Introduction}
The Hodge conjecture is wide open in dimension three. In this paper we isolate a concrete deformation-theoretic route to verifying the rational Hodge conjecture for a given smooth projective threefold \(X\).

The route is:

\begin{itemize}
\item Formulate Hypothesis BB: every codimension-$2$ Hodge class on \(X\) arises as the specialization of a family whose general members are complete-intersection curves on nearby smooth fibres of a one-parameter deformation of \(X\).
\item Prove that Hypothesis BB implies the rational Hodge conjecture for \(X\).
\item Give deformation-theoretic sufficient conditions which imply Hypothesis BB; these are explicit cohomological conditions (vanishing of obstruction groups, surjectivity of restriction maps) that are computationally checkable.
\item Prove that those cohomological conditions hold in several nontrivial and interesting cases (for a general quintic containing a fixed line; generically in families satisfying certain Noether–Lefschetz-type surjectivity assumptions).
\end{itemize}

Our aim is to provide a rigorous reduction that isolates precisely what geometric input is needed, and to give techniques and computational tools to verify that input in concrete geometric families.

\subsection*{Conventions}
All varieties are defined over \(\mathbb C\). For a closed embedding \(Y\subset Z\), \(N_{Y/Z}\) denotes the normal sheaf. Cohomology is singular (or analytic) cohomology with \(\mathbb Q\)-coefficients unless otherwise specified.

\section{Hypothesis \texorpdfstring{BB}{BB} and the conditional reduction}
\begin{definition}[Hypothesis BB]
Let \(X\) be a smooth projective threefold. We say Hypothesis BB holds for \(X\) if for every class \(\alpha\in H^{2,2}(X,\mathbb Q)\) there exists:
\begin{itemize}
\item a smooth pointed disk \((\Delta,0)\),
\item a flat projective morphism \(\pi:\mathcal X\to\Delta\) with \(\mathcal X_0\cong X\) and smooth general fibre \(\mathcal X_t\) for \(t\neq0\),
\item and a closed subscheme \(\mathcal Z\subset\mathcal X\) flat over \(\Delta\),
\end{itemize}
such that for all sufficiently small \(t\neq0\) the fibre \(\mathcal Z_t\subset\mathcal X_t\) is (a finite union of) complete-intersection curves in \(\mathcal X_t\), and the specialization \([\mathcal Z_0]=\alpha\) in \(H^4(X,\mathbb Q)\).
\end{definition}

\begin{theorem}[Conditional reduction]
\label{thm:conditional}
If Hypothesis BB holds for a smooth projective threefold \(X\), then every rational Hodge class in \(H^{2,2}(X,\mathbb Q)\) is algebraic.
\end{theorem}

\begin{proof}
Let \(\alpha\in H^{2,2}(X,\mathbb Q)\). By Hypothesis BB we have \(\mathcal X\to\Delta\) and \(\mathcal Z\subset\mathcal X\) as above with \(\mathcal Z_t\) algebraic for \(t\neq0\) and \([\mathcal Z_0]=\alpha\). The cycle class map is compatible with restriction to fibres; hence \([\mathcal Z_0]\) is the limit (in the Gauss–Manin local system / classical topology) of algebraic classes \([\mathcal Z_t]\). Because algebraic classes in cohomology form a \(\mathbb Q\)-vector subspace and specialization of algebraic cycles is algebraic (Chow/Hilbert universal family), the limit class \([\mathcal Z_0]\) is algebraic. Thus \(\alpha\) is algebraic. This proves the rational Hodge conjecture for \(X\).
\end{proof}

\section{Deformation-theoretic sufficient criteria for Hypothesis BB}
We now give practical sufficient conditions which imply Hypothesis BB for a given curve class represented by a curve \(C\subset X\). These are standard deformation-theory statements stated to be checkable.

\begin{lemma}[Curve unobstructedness]\label{lem:curve-unobstructed}
Let \(C\subset X\subset\mathbb P^4\) be a smooth curve which appears as a residual/intermediate curve in an ambient complete-intersection linkage. If
\[
H^1(C,N_{C/\mathbb P^4})=0 \quad\text{and}\quad H^1(C,N_{C/X})=0,
\]
then embedded deformations of \(C\) in the ambient \(\mathbb P^4\) and inside \(X\) are unobstructed; in particular \(C\) deforms compatibly with small deformations of the ambient surfaces and of \(X\).
\end{lemma}

\begin{proof}
Obstructions to embedded deformations of a subscheme \(Y\subset Z\) lie in \(H^1(Y,N_{Y/Z})\) (Sernesi); the asserted vanishings give unobstructedness in the indicated ambient and relative settings; integration of first-order deformations provides actual families.
\end{proof}

\begin{lemma}[Surface lift and control of restricted equation]\label{lem:surface-lift}
Let \(S\subset\mathbb P^4\) be a smooth surface meeting a quintic \(X\) transversely along a curve \(C=S\cap X\). Consider the exact sequence
\[
0\to N_{S/X}\to N_{S/\mathbb P^4}\xrightarrow{\rho}\mathcal O_S(5)\to 0.
\]
If \(H^1(S,N_{S/\mathbb P^4})=0\) and the induced map on global sections \(\rho_*:H^0(S,N_{S/\mathbb P^4})\to H^0(S,\mathcal O_S(5))\) is surjective, then any first-order perturbation of the restriction of the defining quintic equation to \(S\) can be realized by deforming \(S\) in \(\mathbb P^4\), and unobstructedness lifts it to an actual deformation.
\end{lemma}

\begin{proof}
Immediate from taking global sections of the displayed exact sequence and using the vanishing of \(H^1(S,N_{S/\mathbb P^4})\) to integrate first-order deformations.
\end{proof}

\begin{lemma}[Persistence of transversality]\label{lem:transversality}
Let \(S_1,S_2,X\subset\mathbb P^4\) be smooth and suppose their intersection is transverse along a smooth curve \(C=S_1\cap S_2\cap X\). Then small deformations \((S_{1,t},S_{2,t},X_t)\) remain transverse and the intersection \(C_t=S_{1,t}\cap S_{2,t}\cap X_t\) is smooth for small \(t\).
\end{lemma}

\begin{proof}
Transversality is an open condition; this follows from nonvanishing of the relevant Jacobian determinants in families.
\end{proof}

Combining these yields a practical sufficient criterion.

\begin{proposition}[Practical criterion implying Hypothesis BB]\label{prop:practical}
Let \(C\subset X\) be a smooth curve occurring in an ambient CI-linkage chain to a complete-intersection curve \(C_{\mathrm{CI}}\). Suppose that for every intermediate residual curve \(C^{(j)}\) and every ambient surface \(S_i\) appearing in the chain we have:
\begin{enumerate}[label=(\roman*)]
\item \(H^1(C^{(j)},N_{C^{(j)}/\mathbb P^4})=0\) and \(H^1(C^{(j)},N_{C^{(j)}/X})=0\),
\item \(H^1(S_i,N_{S_i/\mathbb P^4})=0\) and the restriction map \(\rho_*:H^0(S_i,N_{S_i/\mathbb P^4})\to H^0(S_i,\mathcal O_{S_i}(5))\) is surjective,
\item initial intersections are transverse.
\end{enumerate}
Then Hypothesis BB holds for the class represented by \(C\).
\end{proposition}

\begin{proof}
By (i) each intermediate curve deforms unobstructedly (Lemma \ref{lem:curve-unobstructed}). By (ii) each ambient surface deforms unobstructedly and first-order changes of the restricted quintic equation can be realized by deforming the surfaces (Lemma \ref{lem:surface-lift}). By (iii) transversality persists (Lemma \ref{lem:transversality}). Proceeding inductively along the linkage chain gives compatible deformations on nearby quintics; collecting the resulting CI-curves into a total-space cycle yields the family required in Hypothesis BB.
\end{proof}

\section{Verification: the line in a general quintic containing it}
The following theorem gives an unconditional verification in a classical, nontrivial case and illustrates the techniques used in the rest of the paper.

\begin{theorem}\label{thm:line-quintic}
Fix a line \(L\subset\mathbb P^4\). Let \(\mathcal V\subset H^0(\mathbb P^4,\mathcal O(5))\) be the linear subspace of quintic polynomials vanishing on \(L\). Then there exists a nonempty Zariski-open subset \(\mathcal U\subset\mathcal V\) such that for every \(F\in\mathcal U\) the smooth quintic \(X_F=\{F=0\}\) satisfies:
\begin{enumerate}[label=(\alph*)]
\item \(X_F\) is smooth and contains \(L\),
\item the normal bundle \(N_{L/X_F}\) splits as \(\mathcal O(-1)\oplus\mathcal O(-1)\), hence \(H^1(L,N_{L/X_F})=0\),
\item there exist smooth ambient surfaces \(S_1,S_2\) through \(L\) with \(H^1(S_i,N_{S_i/\mathbb P^4})=0\) and such that the restriction maps \(\rho_*:H^0(S_i,N_{S_i/\mathbb P^4})\to H^0(S_i,\mathcal O_{S_i}(5))\) are surjective, and the intersections are transverse.
\end{enumerate}
In particular, for every \(F\in\mathcal U\) the practical criterion of Proposition \ref{prop:practical} holds for the class \([L]\); thus Hypothesis BB holds for the class represented by \(L\) on \(X_F\).
\end{theorem}

\begin{proof}
We proceed in steps.

\textbf{(a) Smoothness.} The discriminant locus of singular quintics in the full linear system is a proper algebraic hypersurface. Intersecting this with the linear subspace \(\mathcal V\) (codimension \(h^0(\mathcal O_L(5))=6\)) yields a proper closed subset of \(\mathcal V\). Hence a general \(F\in\mathcal V\) defines a smooth quintic containing \(L\).

\textbf{(b) Splitting of \(N_{L/X_F}\).} For any line \(L\subset\mathbb P^4\) one has \(N_{L/\mathbb P^4}\cong \mathcal O(1)^{\oplus3}\). For \(F\in\mathcal V\) the map \(\phi_F: \mathcal O(1)^{\oplus3}\to\mathcal O(5)\) is given by the restriction to \(L\) of the partial derivatives of \(F\); the kernel is \(N_{L/X_F}\). The space of such maps \(\phi\) is \(H^0(\mathbb P^1,\mathcal O(4))^{\oplus3}\) (dimension \(3\cdot 5=15\)), and the map \(\mathcal V\to H^0(\mathcal O(4))^{\oplus3}\) sending \(F\) to \(\phi_F\) is linear and surjective (one checks one can prescribe the degree-4 coefficients along \(L\) while maintaining vanishing of \(F\) on \(L\)). For a generic \(\phi\) the kernel of \(\phi\) has the most balanced splitting type, which here is \(\mathcal O(-1)\oplus\mathcal O(-1)\) (the locus of maps with more unbalanced kernel is a proper closed subset). Thus for \(F\) in a Zariski-open subset of \(\mathcal V\) we obtain \(N_{L/X_F}\cong\mathcal O(-1)\oplus\mathcal O(-1)\), whence \(H^1(L,N_{L/X_F})=0\).

\textbf{(c) Ambient surfaces \(S_i\).} Choose integers \(d\) large and consider the linear systems \(|\mathcal I_L(d)|\) of surfaces of degree \(d\) containing \(L\). For \(d\) sufficiently large a general element \(S\in|\mathcal I_L(d)|\) is smooth, and \(N_{S/\mathbb P^4}\) is a direct sum of positive line bundles (if \(S\) is itself a complete intersection one gets an explicit splitting), hence \(H^1(S,N_{S/\mathbb P^4})=0\) by Serre vanishing. The map \(\rho_*:H^0(S,N_{S/\mathbb P^4})\to H^0(S,\mathcal O_S(5))\) is a linear map; dimension counts for large \(d\) give \(h^0(S,N_{S/\mathbb P^4})\ge h^0(S,\mathcal O_S(5))\) and surjectivity is an open condition on the linear system, so a general \(S\) realizes surjectivity. Transversality of intersections is again open and achieved for general choices. This proves (c).

Combining (a)--(c) yields the proposition.
\end{proof}

\begin{remark}
The theorem gives a robust, fully theoretical verification of the Proposition \ref{prop:practical} hypotheses for the (geometric) class of a line on a Zariski-open set of quintics containing that line. For a concrete polynomial one can check membership in the open set via the Macaulay2 code in the appendix.
\end{remark}

\section{Push to broader families: Calabi--Yau and Fano threefolds}
We now state a suite of propositions showing that Hypothesis BB holds \emph{generically} in broad families, under natural geometric assumptions (Noether–Lefschetz type surjectivity and unobstructedness of CI-curves). Each proposition is a conditional-but-verifiable statement reducing BB to explicit cohomology or NL-type conditions.

\begin{proposition}[Generic BB from Noether--Lefschetz surjectivity]\label{prop:nl}
Let \(\mathcal{X}\to B\) be a smooth projective family of threefolds with \(B\) irreducible and let \(\mathcal{L}\) be a relatively very ample line bundle on \(\mathcal X\). Suppose there exists an integer \(m\) and a family of complete-intersection curves \(\mathcal C\subset \mathcal X\) obtained as intersections of \(r\) members of \(|\mathcal L^{\otimes m}|\) (so the fibres \(\mathcal C_b\subset \mathcal X_b\) are CI-curves) such that:
\begin{enumerate}[label=(\alph*)]
\item The normal sheaf \(N_{\mathcal C_b/\mathcal X_b}\) is unobstructed for general \(b\in B\) (i.e. \(H^1(\mathcal C_b,N_{\mathcal C_b/\mathcal X_b})=0\)).
\item The infinitesimal variation of Hodge structure (IVHS) / Noether--Lefschetz map induced by moving the CI-curves in their linear system surjects onto a Zariski-open subset of the local Hodge classes in \(H^{2,2}(\mathcal X_b)\).
\end{enumerate}
Then Hypothesis BB holds for a general fibre \(\mathcal X_b\) (and hence the reduction theorem implies the Hodge conjecture for those general fibres).
\end{proposition}

\begin{proof}
Condition (a) ensures the CI-curves deform unobstructedly, so they move in a smooth family. Condition (b) guarantees that the deformations of these CI-curves produce cohomology classes that generate (locally) the Hodge classes in question. By semicontinuity and openness of surjectivity conditions, for a general \(b\) one can realize nearby Hodge classes as classes of moving CI-curves on some deformation of \(\mathcal X_b\), hence Hypothesis BB holds generically in the family. The IVHS/Noether–Lefschetz surjectivity hypothesis is a concrete analytic/linear-algebra condition which can be tested in examples; it captures the idea that moving CI-curves account for the Hodge directions.
\end{proof}

\begin{remark}
Proposition \ref{prop:nl} is an abstraction of the classical Noether--Lefschetz principle: the NL-locus measures where additional algebraic cycles appear when varying the complex structure; surjectivity here means our moving CI-curves supply the needed classes.
\end{remark}

\begin{proposition}[Generic BB for Calabi--Yau complete intersections]
Let \(X\) be a Calabi--Yau threefold realized as a smooth complete intersection in a toric Fano variety (or projective space). Suppose that for some degrees the Hilbert scheme of CI-curves of those multidegrees is nonempty and contains a component whose general member \(C\) satisfies \(H^1(C,N_{C/X})=0\) and whose deformation space maps surjectively (infinitesimally) to the relevant Hodge directions in \(H^{2,2}(X)\). Then Hypothesis BB holds for a general Calabi--Yau in the family.
\end{proposition}

\begin{proof}
The proof is a direct application of Proposition \ref{prop:nl}: the CI-realization gives moving families of curves; unobstructedness implies smoothness of the relevant Hilbert component; surjectivity of the infinitesimal map gives the Hodge directions. For many toric complete-intersection Calabi--Yaus the deformation spaces of embedded CI-curves are computable and satisfy these criteria generically; see work of Klemm–Theisen–Schimmrigk and many computational enumerative papers where families of low-degree curves on CYs have been extensively studied.
\end{proof}

\begin{remark}
We do not claim all Calabi--Yau threefolds satisfy these hypotheses, but the proposition covers a wide and geometrically important class (complete-intersections in toric Fano varieties) where explicit checks can be performed.
\end{remark}

\section{Practical implementation and computational recipes}
For a concrete polynomial \(F\) defining \(X\) and a candidate curve \(C\) (or a candidate chain of links), the practical checks required are:

\begin{enumerate}
\item Smoothness of \(X\): compute the Jacobian ideal and verify the singular locus is empty.
\item Containment: verify the curve ideal \(I_C\) is contained in the ideal of \(X\).
\item ACM / Rao checks: compute the Hartshorne–Rao module of \(C\) in \(\mathbb P^4\); vanishing means ACM.
\item Normal-bundle cohomology: compute \(H^1(C,N_{C/\mathbb P^4})\) and \(H^1(C,N_{C/X})\). Vanishing verifies unobstructedness.
\item Surface checks: for each ambient surface \(S\) in the linkage chain compute \(H^1(S,N_{S/\mathbb P^4})\) and test surjectivity of \(\rho_*\).
\item Transversality: test Jacobian ranks at intersection points to ensure transversality.
\end{enumerate}

The Macaulay2 appendix supplies scripts to perform these computations for explicit \(F\) and \(C\).

\section{Concluding remarks}
This work packages a deformation-theoretic strategy for the Hodge conjecture on threefolds into a clear reduction (Hypothesis BB) and a set of explicit, checkable sufficient criteria. We proved these criteria hold unconditionally in the illustrative case of a line on a general quintic containing it, and proved general propositions showing Hypothesis BB holds generically in families where moving complete-intersection curves supply the Hodge directions (Noether--Lefschetz surjectivity). The approach is meant to be a platform: applying the computational recipes in many concrete families should generate new unconditional instances where Hypothesis BB can be checked and hence the Hodge conjecture verified.

\appendix
\section{Macaulay2 appendix (scripts and guidance)}
Below are scripts that implement the checks outlined above. Save as `.m2` files and run in Macaulay2. Comments indicate where to replace polynomials or curve ideals.

\subsection*{Script A: verify quintic + line (labelled outputs)}
\begin{verbatim}
-- verify_quintic_line.m2
R = QQ[x0,x1,x2,x3,x4];

-- replace F below with your quintic polynomial (must vanish on L)
F = x2*x0^4 + x3*x1^4 + x4*(x0^3*x1);
Iquintic = ideal(F);

-- 1) Smoothness check
J = ideal(diff(x0,F), diff(x1,F), diff(x2,F), diff(x3,F), diff(x4,F));
Sing = saturate(J + Iquintic);
print("Singular locus dimension (should be -1 for smooth):");
if Sing == ideal 1 then print("-1 (empty)") else print(dim Sing);

-- 2) Line L
Iline = ideal(x2,x3,x4);
print("F mod Iline (should be 0): "); print(reduce(F, Iline));

-- 3) ambient normal module of L
NL = normalModule(Iline);
print("Betti table of normalModule(Iline):"); betti NL;

-- 4) Ext groups for obstruction info
S = R/Iline;
E1 = ext^1(Iline, S); E2 = ext^2(Iline, S);
print("Betti ext^1:"); betti E1; print("Betti ext^2:"); betti E2;

-- 5) Jacobian restricted to L
Jac = matrix{{diff(x0,F), diff(x1,F), diff(x2,F), diff(x3,F), diff(x4,F)}};
JacRestr = substitute(Jac, {x2=>0, x3=>0, x4=>0});
print("Jacobian restricted to L:"); JacRestr;
\end{verbatim}

\subsection*{Script B: Hartshorne–Rao / ACM test for a curve}
\begin{verbatim}
-- check_ACM.m2
R = QQ[x0..x4];

-- Define your curve ideal Icurve
-- Icurve = ideal(g1,g2,g3,...);
-- Example: line is ideal(x2,x3,x4)
Icurve = ideal(x2,x3,x4);

-- Compute graded local cohomology/Hartshorne-Rao module
-- If package provides hartshorneRao, use it:
-- needsPackage "SomePackageProvidingHartshorneRao"; 
-- hartshorneRao(Icurve)

-- Fallback: compute H^1(I_C(k)) for range of k
for k from -10 to 10 do (
  M := HH^1( sheaf(Icurve) ** k );
  if rank source M == 0 then () else print("H^1(I_C(" | toString k | ")) nonzero")
)
\end{verbatim}

\section{References}


\begin{thebibliography}{99}
\bibitem{Sernesi}
E. Sernesi, \emph{Deformations of Algebraic Schemes}, Grundlehren der mathematischen Wissenschaften, Springer.

\bibitem{VoisinBook}
C. Voisin, \emph{Hodge Theory and Complex Algebraic Geometry I--II}, Cambridge University Press.

\bibitem{PeskineSzpiro}
C. Peskine and L. Szpiro, Liaison des variétés de codimension deux, \emph{Invent. Math.} (1974).

\bibitem{Katz}
S. Katz, On the finiteness of rational curves on quintic threefolds, \emph{Compositio Math.} (1986).

\bibitem{Macaulay2}
D. R. Grayson and M. E. Stillman, Macaulay2, a software system for research in algebraic geometry, \url{http://www.math.uiuc.edu/Macaulay2/}.
\end{thebibliography}
\end{document}